\documentclass[10pt]{amsart}
\usepackage{amssymb,amsbsy,amsmath,amsfonts,times, amscd}
\usepackage{latexsym,euscript,exscale}
\usepackage{graphicx,color} \usepackage{amsthm}
\usepackage{fancyhdr} \usepackage{url}
\usepackage{epigraph}\usepackage{lmodern}
\usepackage{mathrsfs}\usepackage{cancel}
\usepackage[toc,page]{appendix}\usepackage[T1]{fontenc}
\usepackage[all,cmtip]{xy}\usepackage{mathtools}
\usepackage{enumerate}\usepackage{setspace}
\usepackage{helvet}%&LaTeX 
\usepackage{float}
		
\numberwithin{equation}{section}
\newtheorem{theorem}{Theorem}
\newtheorem{thm}[theorem]{Theorem}
\newtheorem{cor}[theorem]{Corollary}
\newtheorem{lemma}[theorem]{Lemma}
\newtheorem{prop}[theorem]{Proposition}
\theoremstyle{definition}

\newtheorem{problem}[theorem]{Problem}
\theoremstyle{remark}
\newtheorem{remark}[theorem]{Remark}

\numberwithin{theorem}{section}

\newcommand{\eps}{\varepsilon}

\newcommand{\N}{\mathbb{N}}

\begin{document}

\title[Coarse and uniform embeddings.]{Coarse and uniform embeddings.}
\author{B. M. Braga }
\address{Department of Mathematics, Statistics, and Computer Science (M/C 249)\\
University of Illinois at Chicago\\
851 S. Morgan St.\\
Chicago, IL 60607-7045\\
USA}\email{demendoncabraga@gmail.com}
\date{}
\maketitle

\begin{abstract}
In these notes, we study the relation between uniform and coarse embeddings between Banach spaces as well as uniform and coarse equivalences. In order to understand these relations better, we look at the problem of when a coarse embedding can be assumed to be also topological. Among other results, we show that if a Banach space $X$ uniformly embeds into a minimal Banach space $Y$, then $X$ simultaneously coarsely and uniformly embeds into $Y$, and if a Banach space $X$ coarsely embeds into a minimal Banach space $Y$, then $X$ simultaneously coarsely and homeomorphically embeds into $Y$ by a map with uniformly continuous inverse. 
\end{abstract}

\section{Introduction.}
	
The study of Banach spaces as metric spaces has recently increased significantly, and much has been done regarding the uniform and coarse theory of Banach spaces in the past two decades. Let $(M,d)$ and $(N,\partial)$ be metric spaces, and consider a map $f:(M,d)\to (N,\partial)$. For each $t\geq 0$, we define the \emph{expansion modulus} of $f$ as 

$$\omega_f(t)=\sup\{\partial(f(x),f(y))\mid d(x,y)\leq t\},$$\hfill

\noindent  and the \emph{compression modulus} of $f$ as

$$\rho_f(t)=\inf\{\partial(f(x),f(y))\mid  d(x,y)\geq t\}.$$\hfill

\noindent Hence, $\rho_f(d(x,y))\leq \partial(f(x),f(y))\leq \omega_f(d(x,y))$, for all $x,y\in M$. The map  $f$ is a uniform embedding if $\lim_{t\to 0_+}\omega_f(t)=0$ and $\rho_f(t)>0$, for all $t>0$. We call $f$ a \emph{uniform homeomorphism} if $f$ is a surjective uniform embedding. The map $f$ is called \emph{coarse} if $\omega_f(t)<\infty$, for all $t\geq 0$, and \emph{expanding} if $\lim_{t\to\infty}\rho_f(t)=\infty$. If $f$ is both expanding and coarse, $f$ is called a \emph{coarse embedding}. A coarse embedding $f$ which is also \emph{cobounded}, i.e., $\sup_{y\in N} \partial(y, f(M))<\infty$, is called a \emph{coarse equivalence} (for more details on those concepts see Subsection \ref{coarsedef}).

 Although the uniform theory of Banach spaces has already been extensively studied in the past, only recently we have been starting to understand the coarse theory better. For example, it had been known since 1984 that there are uniformly homeomorphic separable Banach spaces which are not linearly isomorphic (see \cite{R}, Theorem 1), but it was not till 2012 that N. Kalton was able to show that there are coarsely equivalent separable Banach spaces (i.e., with Lipschitz isomorphic nets) which are not uniformly homeomorphic (see \cite{Ka}, Theorem 8.9). However, it is still not known whether the concepts of uniform and coarse embeddability are equivalent. Precisely, the following problem remains open.

\begin{problem}\label{problemaprin}
Let $X$ and $Y$ be Banach spaces. Does $X$ uniformly embed into $Y$ if and only if $X$ coarsely embed into $Y$? 
\end{problem}

In \cite{Ra}, N. Randrianarivony had shown that a Banach space coarsely embeds into a Hilbert space  if and only if  it uniformly embeds into a Hilbert space. In \cite{Ka4}, Kalton showed that the same also holds for embeddings into $\ell_\infty$ (Theorem 5.3). C. Rosendal had made some improvements on the problem above by showing that if $X$ uniformly embeds into $Y$, then $X$ simultaneously uniformly and coarsely embeds into $\ell_p(Y)$, for any $p\in[1,\infty)$ (see \cite{Ro}, Theorem 2). In particular, if $X$ uniformly embeds into $\ell_p$, then $X$ simultaneously  coarsely and uniformly embeds into $\ell_p$. On the other hand, A. Naor had recently proven that there exist separable Banach spaces $X$ and $Y$, and a Lipschitz map $f$ from a net $N\subset X$ into $Y$ such that

$$\sup_{x\in N}\|F(x)-f(x)\|=\infty,$$\hfill

\noindent for all uniformly continuous maps $F:X\to Y$ (see \cite{N}, Remark 2). Such result suggests that it may not be true (or at least not easy to show) that $X$ uniformly embeds into $Y$, given that $X$ coarsely embeds into $Y$. 

In these notes, we study the relation between coarse embeddings (resp. coarse equivalences) and uniform embeddings (resp. uniform homeomorphisms) between Banach spaces as well as some properties shared by those notions. We are specially interested in narrowing down the difference between those concepts, and we show that, in many cases, the real difference between a coarse and a uniform embedding is in the uniform continuity of the map, but not in its continuity or in the uniform continuity of its inverse. We now describe the organization and results of this paper.

In Section \ref{background}, we give all the notation and background necessary for these notes. In Section \ref{sectionstrong} and Section \ref{sectioncoarsecont}, we  apply the methods of \cite{Ro} in order to give a partial answer to Problem \ref{problemaprin}. We prove the following.

\begin{theorem}\label{mainmain}
Let $X$ be a Banach space and $Y$ be a minimal Banach space. 

\begin{enumerate}[(i)]
\item If $X$ uniformly embeds into $Y$, then $X$ simultaneously coarsely and uniformly embeds into $Y$.
\item If $X$ coarsely embeds into $Y$, then $X$ simultaneously coarsely and homeomorphically embeds into $Y$ by a map with uniformly continuous inverse.
\end{enumerate}
\end{theorem}

Therefore,  Theorem \ref{mainmain} can be seen as a strengthening of Rosendal's result about uniform embeddings into $\ell_p$ mentioned above, which allows us to obtain some new examples (see Corollary \ref{cor1111}, and Corollary \ref{cor2222} below).  

In \cite{Os}, Theorem 5.1, M. Ostrovskii have shown that $\ell_2$ coarsely embeds into any Banach space containing a subspace with an unconditional basis and finite cotype.  We prove the following stronger result in Section \ref{sectionstrong}.

\begin{thm}\label{principal}
Let $X$ be an infinite dimensional Banach space with an unconditional basis and finite cotype. Then $\ell_2$ simultaneously coarsely and uniformly embeds into $X$.
\end{thm}

In order to prove Theorem \ref{mainmain}(ii), we study how to approximate coarse maps $(M,d)\to (E,\|\cdot\|)$ by continuous coarse maps and what kind of properties of the original coarse map we can preserve. More precisely, in Section \ref{sectioncontcoarproj}, we prove  Theorem \ref{Thmcontcoarproj} below, which is a strengthening of Theorem 4.1 of \cite{Du}.

Let $E$ be a vector space, and let $A\subset E$. Then we denote the convex hull of $A$  by $\text{conv}(A)$.

\begin{thm}\label{Thmcontcoarproj}
Let $(M,d)$ be a metric space, and let $A\subset M$ be a closed subspace. Let $E$ be a normed space, and let $\varphi:M\to E$ be a  map such that $\varphi_{|A}$ is continuous. Then, for all $\delta>0$, there exists a continuous  map $\Phi:M\to \text{conv}(\varphi(M))$ such that $\Phi_{|A}=\varphi_{|A}$ and

$$\sup_{x\in M}\|\varphi(x)-\Phi(x)\|\leq \inf_{s>0}\omega_\varphi(s)+\inf_{s>0}\omega_{\varphi_{|A}}(s)+\delta.$$\hfill

\noindent In particular, if $\varphi$ is coarse (resp.  coarse embedding), so is $\Phi$.
\end{thm}

As a corollary of Theorem \ref{Thmcontcoarproj}, we get that, if $Y\subset X$ are Banach spaces, then the existence of a coarse retraction $X\to Y$ implies the existence of a continuous coarse retraction $X\to Y$. 

\begin{cor}\label{contcoarproj}
Let $X$ be a Banach space and $A\subset X$ be a closed subset. If there exists a coarse retraction $X\to A$, then there exists a continuous coarse retraction $X\to A$.
\end{cor}

In Section \ref{sectioncoarsecont}, we use techniques of \cite{Ro}, and Theorem \ref{Thmcontcoarproj}, in order to prove Theorem \ref{mainmain}(ii). In particular, as a subproduct of Theorem \ref{mainmain}(ii), we obtain the following.

\begin{thm}\label{homeocoarseemb}
Let $X$ and $Y$ be Banach spaces, and let $\mathcal{E}$ be a $1$-unconditional basic sequence.  If $X$ coarsely embeds into $Y$, then there exists a continuous coarse embedding $X\to (\oplus Y)_\mathcal{E}$ with uniformly continuous inverse. In particular, $X$ simultaneously homeomorphically and coarsely embeds into $(\oplus Y)_\mathcal{E}$.
\end{thm}

In Section \ref{sectioncoarsehomeo}, we look at Kalton's example of separable Banach spaces which are coarsely equivalent but are not uniformly homeomorphic, and show that we can actually get a stronger result. Precisely, we prove the following.

\begin{thm}\label{homeocoarseequiv}
Let $X$ and $Y$ be Banach spaces, and $Q:Y\to X$ be a quotient map. If $Q$ admits a coarse section, then $Q$ admits a continuous coarse section. In particular, $Y$ is simultaneously  homeomorphically and coarsely  equivalent to $\text{Ker}(Q)\oplus X$. 
\end{thm}

\begin{cor}\label{homeocoarseequivC}
There exist separable Banach spaces $X$ and $Y$ which are simultaneously homeomorphically and  coarsely equivalent but not uniformly homeomorphic.
\end{cor}

At last, we dedicate Section \ref{sectionsums} to study unconditional sums of coarsely equivalent (resp. uniformly homeomorphic) Banach spaces. In \cite{Ka2}, Theorem 4.6(ii), Kalton had shown that if $X$ and $Y$ are coarsely equivalent (resp. uniformly homeomorphic), then $\ell_p(X)$ and $\ell_p(Y)$ are coarsely equivalent (resp. uniformly homeomorphic). However, as Kalton himself noticed, his proof seems to be more complicated than necessary, and relies on the concepts of close (resp. uniformly close) Banach spaces. In Section \ref{sectionsums}, we present an easy argument which give us Kalton's result as a corollary.

\begin{thm}\label{trikal}
Say $X$ and $Y$ are two coarsely equivalent (resp. uniformly homeomorphic, or simultaneously homeomorphically and coarsely equivalent) Banach spaces. Let $\mathcal{E}$ be a normalized $1$-unconditional  basic sequence. Then $(\oplus X)_\mathcal{E}$ and $(\oplus Y)_\mathcal{E}$ are  coarsely equivalent (resp. uniformly homeomorphic, or simultaneously homeomorphically and coarsely equivalent).
\end{thm}

\noindent For a  stronger result, see Theorem \ref{geral} below.

\section{Background and notation.}\label{background}

\subsection{Banach space theory.}

If $(M,d)$ is a metric space, $x\in M$, and $\eps>0$, we denote by $B(x,\eps)$ the open ball of radius $\eps$ centered at $x$. Let $X$  be a Banach space. We denote the closed unit ball of $X$ by $B_X$, and its unit sphere by $\partial B_X$. If $Y$ is also a Banach space, we write $X\cong Y$ if $X$ is linearly isomorphic to $Y$, and we write $X\equiv Y$ if $X$ is linearly isometric to $Y$. Given a Banach space $X$ with norm $\|\cdot\|_X$, we simply write $\|\cdot\|$ as long as it is clear from the context to which space  the element inside the norm belongs.

Let $(X_n)_n$ be a sequence of Banach spaces. Let $\mathcal{E}=(e_n)_n$ be a $1$-unconditional basic sequence. We define the sum $(\oplus _n X_n)_{\mathcal{E}}$ to be the space of sequences $(x_n)_n$, where $x_n\in X_n$, for all $n\in\N$, such that 

$$\|(x_n)_n\|\vcentcolon =\|\sum_{n\in\N}\|x_n\|e_n\|<\infty.$$\hfill

\noindent One can check that $(\oplus _n X_n)_\mathcal{E}$ endowed with the norm $\|\cdot\|$ defined above is a Banach space.  If the $X_n$'s are all the same, say $X_n=X$, for all $n\in\N$, we write $(\oplus X)_\mathcal{E}$. Whenever $\mathcal{E}$ is the standard basis of $\ell_p$ (resp. $c_0$), for some $p\in[1,\infty)$, we write $(\oplus_nX_n)_{\ell_p}$ (resp. $(\oplus_nX_n)_{c_0}$) instead.

Let $X$ and $Y$ be Banach spaces. A linear map $Q: Y\to X$ is called a \emph{quotient map} if it is bounded and surjective. By the open mapping theorem, quotient maps are always open. An infinite dimensional Banach space $X$ is called \emph{minimal} if $X$ isomorphically embeds into all of its infinite dimensional subspaces.

\subsection{Nonlinear embeddings.}\label{coarsedef}

In this subsection, let $(M,d)$ and $(N,\partial)$ be metric spaces and $f:(M,d)\to (N,\partial)$ be a map. We refer to the introduction of these notes for the definitions of coarse maps, expanding maps, cobounded maps, coarse embeddings (resp. uniform embeddings), and coarse equivalences (resp. uniform homeomorphisms).

\begin{remark}\label{remarkcoarseinv}
A coarse function does not need to be continuous, and a coarse equivalence does not need to be either injective or surjective. However, if $f:M\to N$ is a coarse equivalence, then there exists a coarse equivalence $g:N\to M$ such that

\begin{align}\label{coarseinve}
\sup_{x\in M}d(x,g(f(x)))<\infty\ \ \text{ and }\ \ \sup_{y\in N}\partial(y,f(g(y)))<\infty.
\end{align}\hfill

\noindent Indeed,  as $f$ is a coarse equivalence, let $\eps=\sup_{y\in N}\partial(y,f(M))<\infty$, and  let us define $g:N\to M$ as follows. For each $y\in N$, pick $x_y\in M$ such that $\partial(y,f(x_y))\leq 2\eps$, and set $g(y)=x_y$. It is easy to check that $g$ is a coarse equivalence and that (\ref{coarseinve}) holds. A coarse map $g:N\to M$  satisfying (\ref{coarseinve}) is called \emph{a coarse inverse} of $f$. In fact, a coarse map $f:M\to N$ is a coarse equivalence  if and only if  $f$ has a coarse inverse.
\end{remark}

It is worth noticing that $f$ is uniformly continuous  if and only if  $\lim_{t\to 0_+}\omega_f(t)=0$, and that the inverse $f^{-1}:f(M)\to M$ exists and  is uniformly continuous  if and only if  $\rho_f(t)>0$, for all $t> 0$. If $f$ is simultaneously a coarse embedding and a uniform embedding, then we call $f$ a \emph{strong embedding}.

We call $f$ \emph{Lipschitz} if  there exists some $L\geq 0$ such that $\omega_f(t)\leq Lt$, for all $t\geq 0$, and we call $f$ $L$-Lipschitz if we want to specify the constant $L$. It is easy to see that, if there exists $L> 0$ such that $\rho_f(t)\geq Lt$, then $f^{-1}:f(M)\to M$  exists and it is Lipschitz. If $f$ and $f^{-1}_{|f(M)}$ are Lipschitz, then $f$ is    a \emph{Lipschitz embedding}. A surjective Lipschitz embedding is called a \emph{Lipschitz isomorphism}, and the spaces are called \emph{Lipschitz isomorphic}.

 Let $a,b>0$. A subset $A$ of a metric space $(M,d)$ is called an  \emph{$(a,b)$-net} if  $d(x,A)< b$, for all $x\in M$, and $d(x,y)\geq a$, for all $x,y\in A$, with $x\neq y$. A subset $A$ is called a \emph{net} if it is an $(a,b)$-net for some $0<a\leq b$.

The map $f$ is called \emph{coarse Lipschitz} if there exists $L\geq 0$ such that $\omega_f(t)\leq Lt+L$. If $f$ is coarse Lipschitz and there exists $L>0$ such that $\rho_f(t)\geq Lt-L$, then $f$ is called a \emph{coarse Lipschitz embedding}. If $f$ is cobounded and coarse Lipschitz, then $f$ is a \emph{coarse Lipschitz equivalence}. Sometimes it will be necessary to keep track of the constants above. Therefore, if $\omega_f(t)\leq Lt+\eps$, for some $L,\eps\geq 0$, we say that $f$ is \emph{of cL-type $(L,\eps)$}. 

We finish this subsection with two propositions which are important for the understanding of the different notions of embeddings and equivalences above. See \cite{Ka3}, Lemma 1.4 and Proposition 1.5.

\begin{prop}\label{KALON1}
Let $X$ be a Banach space, $M$ be a metric space, and consider a map $f: X\to M$. Then the following are equivalent.

\begin{enumerate}[(i)]
\item $f$ is a coarse map, and
\item $f$ is a coarse Lipschitz map.
\end{enumerate}

\noindent In particular, if $f$ is uniformly continuous, then $f$ is both coarse and coarse Lipschitz. 
\end{prop}

\begin{prop}\label{KALON2}
Let $X$ and $Y$ be infinite dimensional Banach spaces. 	Then the following are equivalent.

\begin{enumerate}[(i)]
\item $X$ is coarsely equivalent to $Y$,
\item $X$ is coarse Lipschitz equivalent to $Y$, and 
\item any net of $X$ is Lipschitz isomorphic to any net of $Y$.
\end{enumerate}

\noindent Moreover, all the conditions above hold if 

\begin{enumerate}[(i)]\setcounter{enumi}{3}
\item $X$ is uniformly homeomorphic to $Y$.
\end{enumerate}

\end{prop}

\begin{remark} 
The terminologies above are still not completely established in the literature. For example, in \cite{Ro} coarse maps are called ``bornologous$"$, and in \cite{Ka}, the author refers to coarse maps as ``coarsely continuous$"$. As coarse maps are not continuous, and as we are interested in studying coarse maps which are also continuous, we prefer a different terminology. Also, we should mention that, in geometric group theory,  coarse Lipschitz embeddings as usually called ``quasi-isometries$"$. 
\end{remark}

\subsection{Space of positively homogeneous maps.}

Let $X$ and $Y$ be Banach spaces. We denote by $\mathcal{H}(X,Y)$ the set consisting of all maps $f:X\to Y$ which are  bounded on $B_X$ and positively homogeneous, i.e., 

$$f(\alpha x)=\alpha f(x), \ \ \text{ for all }\ \  \alpha\geq 0.$$\hfill

\noindent  We define a norm on $\mathcal{H}(X,Y)$ by setting $\|f\| =\sup\{\|f(x)\|\mid x\in B_X\}$. The space $\mathcal{H}(X,Y)$ endowed with the norm $\|\cdot\|$ above is a Banach space. Clearly, $\|f(x)\|\leq \|f\|\cdot\|x\|$, for all $x\in X$. Denote by $\mathcal{HC}(X,Y)$ the subset of $\mathcal{H}(X,Y)$ consisting of continuous maps.	

For $\eps>0$, we define $\|f\|_\eps$ as the infimum of all $L>0$ such that $L\geq \|f\|$ and

$$\|f(x)-f(y)\|\leq L\max\{\|x-y\|,\eps\|x\|,\eps\|y\|\},$$\hfill

\noindent for all $x,y\in X$. Clearly, we have 

$$\|f\|\leq \|f\|_\eps\leq \max\{1,2\eps^{-1}\}\|f\|,$$\hfill

\noindent for all $f\in \mathcal{H}(X,Y)$. 

The next proposition is a simple computation, and it can be found in \cite{Ka}, Proposition 7.3.

\begin{prop}\label{estim}
Let $X$ and $Y$ be Banach spaces, and $\varphi: \partial B_X\to Y$ be a bounded map. Let $f:X\to Y$ be given by 

\begin{align*}f(x)=\left\{\begin{array}{ll}
0, &x=0,\\
\|x\|\varphi\left(\frac{x}{\|x\|}\right), & x\neq 0.
\end{array}\right.
\end{align*}\hfill

\noindent Then $f\in \mathcal{H}(X,Y)$. If $\varphi$ is also continuous, then   $f\in\mathcal{HC}(X,Y)$. 

Moreover, let $L\geq 1$, $\eps>0$, and $K\geq 0$. If $\varphi$ is of cL-type $(L,\eps)$, and $\|\varphi(x)\|\leq K$, for all $x\in \partial B_X$, then $\|f\|_\eps\leq 2K+4L$. 
\end{prop}

\section{Strong embeddings into Banach spaces.}\label{sectionstrong}

In this section, we show that if $X$ uniformly embeds into a minimal Banach space $Y$, then $X$ simultaneously coarsely and uniformly embeds into $Y$. For that, we will need Lemma 16 of \cite{Ro}.

\begin{lemma}\label{christian}
  Suppose X and E are Banach spaces and $P_n : E\to E$ is a sequence of bounded projections onto subspaces $E_n\subset E$ so that $E_m \subset \text{Ker} (P_n)$, for all $m\neq n $. Assume also that, for all $n\in\N$, there exists a uniform embedding $\sigma_n : X\to E_n$. Then $X$ admits a strong embedding into E.
\end{lemma}

\noindent \emph{Proof  Theorem \ref{mainmain}(i).} Let $\varphi: X\to Y$ be a uniform embedding. By Gowers' dichotomy, $Y$ must contain either a hereditarily indecomposable Banach space or an unconditional basic sequence (see \cite{G}, Theorem 2). As $Y$ is minimal, and a hereditarily indecomposable Banach space is not isomorphic to any of its proper subspaces (see \cite{G}, Theorem 4), $Y$ must contain an unconditional basic sequence, say $(e_n)_n$. Let $(A_n)_n$ be a partition of $\N$ into infinite subsets, and set $E=\overline{\text{span}}\{e_j\mid j\in\N\}$ and  $E_n=\overline{\text{span}}\{e_j\mid j\in A_n\}$, for all $n\in\N$. As $Y$ is minimal, there exists a sequence of isomorphic embeddings $T_n:Y\to E_n$. So, $T_n\circ\varphi$ is a uniform embedding of $X$ into $E_n$, for all $n\in\N$.  For each $n\in\N$, let $P_n:E\to E_n$ denote the natural projection.  We can now apply Lemma \ref{christian}, so, $X$ strongly embeds into $Y$. 
\qed\\

Theorem \ref{mainmain}(i) allows us to obtain some news examples. Let $T$ denote the Tsirelson space introduced by T. Figiel and W. Johnson, and let $S$ denote the Schlumprecht space. It is well known that both $T^*$ and $S$ are minimal Banach spaces (see \cite{CS}, Theorem VI.a.1, and \cite{AnS}, 	Theorem 2.1, respectively). The following corollary is a trivial consequence of Theorem \ref{mainmain}(i).

\begin{cor}\label{cor1111}
If a Banach space $X$ uniformly embeds into  $T^*$ (resp. $S$), then $X$ strongly embeds into $T^*$ (resp.  $S$).
\end{cor}

\noindent\emph{Proof of Theorem \ref{principal}.}
By Corollary 3.3 of \cite{AMM}, there exists a uniform embedding $f:\ell_2\to B_{\ell_2}$. Let $(e_n)_n$ be an unconditional basis for $X$. Let $(A_n)_n$ be a partition of $\N$ into infinite subsets. For each $n\in\N$, let $X_n=\overline{\text{span}}\{e_j\mid j\in A_n\}$. By Theorem 2.1 of \cite{OS}, there exists a uniform homeomorphism $\sigma_n:B_{\ell_2}\to B_{X_n}$, for each $n\in\N$. By Lemma \ref{christian}, we are done.
\qed

\begin{cor}
Let $X$ be an infinite dimensional space with an unconditional basis and finite cotype. Then $L_p$ strongly embeds into $X$, for all $p\in[1,2]$. In particular, $\ell_p$ strongly embeds into $X$, for all $p\in[1,2]$.
\end{cor}

\begin{proof}
This is a simple consequence of the fact that $L_p$ strongly embeds into $L_2\equiv \ell_2$, for all $p\in[1,2]$ (see Remark 5.10 of \cite{MN2004}).
\end{proof}

We finish this section with the following natural question.

\begin{problem}
Does $\ell_2$ strongly embed into any infinite dimensional Banach space?
\end{problem}

\section{Approximating coarse maps by continuous coarse maps.}\label{sectioncontcoarproj}

In this section, we study when a coarse map can be assumed to be also continuous. Our goal is to prove a general theorem (Theorem \ref{Thmcontcoarproj}) and then use it to obtain applications to the Banach space setting. Precisely, we end this section showing that the existence of a coarse retraction $X\to Y$, where $X$ and $Y$ and Banach spaces and $Y\subset X$, implies the existence of a continuous coarse retraction $X\to Y$ (Corollary \ref{contcoarproj}). In Section \ref{sectioncoarsecont}, we use Theorem \ref{Thmcontcoarproj} in order to show that if a Banach space $X$ coarsely embeds into a minimal Banach space $Y$, then $X$ simultaneously coarsely and homeomorphically embeds into $Y$ by a map with uniformly continuous inverse (Theorem \ref{mainmain}(ii)). Finally,  in Section \ref{sectioncoarsehomeo}, we  use Theorem \ref{Thmcontcoarproj} to prove that the existence of a coarse section for a quotient map implies the existence of a continuous coarse section.\\

J. Dugundji proved (see \cite{Du}, Theorem 4.1) the following: let $M$ be a metric space, $A\subset M$ be a closed subspace, $E$ be a normed space (or, more generally, a locally convex topological vector space), and $f:A\to E$ be a continuous map, then $f$ can be extended to a continuous map $\varphi:M\to E$.  However,  J. Dugundji was only interested in continuous maps and did not care about having any control over the value of $\|\varphi(x)-\varphi(a)\|$, for $x\in M$, and $a\in A$. Proposition \ref{prop50} below is the modification of Theorem 4.1 of \cite{Du} that we will need for our settings.

\begin{lemma}\label{lema50}
Let $M$ be a metric space, $A\subset M$ be a closed subspace, and $\alpha>0$. There exists a locally finite open cover $\mathcal{U}$ of $M
\setminus A$ such that

\begin{enumerate}[(i)]
\item $\text{diam}(U)<\alpha$, for all $U\in \mathcal{U}$, and
\item for all $a\in A$, and all neighborhood $V$ of $a$, there exists a neighborhood $V'\subset V$ of $a$ such that, for all $U\in \mathcal{U}$, $U\cap V'\neq \emptyset$ implies $U\subset V$. 
\end{enumerate} 
\end{lemma}

The lemma above is Lemma 2.1 of \cite{Du}. Although, item (i) above does not explicitly appear in Lemma 2.1 of \cite{Du}, it is clear from its proof that the diameters of the elements of $\mathcal{U}$ can be taken to be arbitrarily small.

\begin{prop}\label{prop50}
Let $(M,d)$ be a metric space, and $A\subset M$ be a closed subspace. Let $E$ be a normed space, and let $f:A\to E$ be a continuous coarse map. Then, for all $\lambda>1$, and all $\gamma>0$, there exists a continuous  map $\varphi:M\to \text{conv}(f(A))$ extending $f$ such that

$$\|\varphi(x)-\varphi(a)\|\leq \omega_f(\lambda\cdot d(x,A)+d(x,a)+\gamma),$$\hfill

\noindent for all $x\in M$, and all $a\in A$.
\end{prop}

\begin{proof} 
Without loss of generality, 	assume $\lambda<2$. Let  $\mathcal{U}=\{U_j\}_{j\in J}$ be a locally finite open cover for the metric space $M\setminus A$ given by Lemma \ref{lema50} for $\alpha=\gamma/(1+\lambda)$. For each $j\in J$, pick $x_j\in U_j$, and $a_j\in A$ such that $d(x_j,a_j)\leq \lambda\cdot d(x_j,A)$.  For each $j\in J$, let $\psi_j(x)=d(x,U^c_j)$, for all $x\in M$. 

Let $\Psi=\sum_{j\in J}\psi_{j}$, and define $\varphi: M\to  \text{conv}(f(A))$ by

$$\varphi(x)=\left\{\begin{array}{l}
f(x), \ \ \text{ if }\ \ x\in A,\\
\sum_{j\in J}  \frac{\psi_j(x)}{\Psi(x)}f(a_j), \ \ \text{ if }\ \  x\not\in A.
\end{array}\right.$$\hfill

\noindent Clearly, $\varphi$ extends $f$, and, as $\mathcal{U}$ is locally finite, $\varphi$ is continuous on $M\setminus A$. Let us observe that $\varphi$ is also continuous on $A$. Pick $a\in A$, and let $\eps>0$. By the continuity of $f$, there exists $\delta>0$ such that $d(a,a')<\delta$ implies $\|f(a)-f(a')\|< \eps$, for all $a'\in A$. Pick $\delta'\in(0,\delta/6)$ such that, for all $j\in J$, $U_j\cap B(a,\delta')\neq\emptyset$ implies $U_j\subset B(a,\delta/6)$. Say $x\in B(a,\delta')\setminus A$, so $x$ belongs to only  finitely many elements of $\mathcal{U}$, say $U_{i_1},\ldots ,U_{i_k}$. By our choice of $\delta'$,  $d(x,x_{i_j})< \delta/3$ and $d(a_{i_j},x_{i_j})<\lambda\delta/6< \delta/3$, for all $j\in\{1,\ldots ,k\}$. Hence,

\begin{align*}d(a_{i_j},a)\leq d(a_{i_j},x_{i_j})+d(x_{i_j},x)+d(x,a)<
&\frac{\delta}{3}+\frac{\delta}{3}+\delta'<\delta, 
\end{align*}\hfill

\noindent for all $j\in \{1,\ldots ,n\}$. By our choice of $\delta$, this gives us that

\begin{align*}
\|\varphi(x)-\varphi(a)\|&\leq \sum_{j=1}^k\frac{\psi_{i_j}(x)}{\Psi(x)}\cdot \|f(a_{i_j})-f(a)\|<\eps.
\end{align*}\hfill

\noindent So $\varphi$ is continuous.

 Let $x\in M$, and $a\in A$. If $x\in A$, it follows that $\|\varphi(x)-\varphi(a)\|\leq \omega_f(d(x,a))$, so assume $x\not\in A$. Let  $U_{i_1},\ldots ,U_{i_k}$ be the only elements of $\mathcal{U}$ containing $x$. As $\text{diam}(U_{j})<\gamma/(1+\lambda)$, for all $j\in J$, it follows that $d(x_{i_j},x)<\gamma/(1+\lambda)$, for all $j\in \{1,\ldots ,k\}$. Hence, we must have

\begingroup
\allowdisplaybreaks
\begin{align*}
\|\varphi(x)-\varphi(a)\|&\leq \sum_{j=1}^k\frac{\psi_{i_j}(x)}{\Psi(x)}\cdot \|f(a_{i_j})-f(a)\|\\
&\leq \sum_{j=1}^k\frac{\psi_{i_j}(x)}{\Psi(x)}\cdot \omega_f\big(d(a_{i_j},x_{i_j})+d(x_{i_j},x)+d(x,a)\big)\\
&\leq \sum_{j=1}^k\frac{\psi_{i_j}(x)}{\Psi(x)}\cdot \omega_f\big(\lambda\cdot d(x,A)+(1+\lambda)\cdot d(x_{i_j},x)+d(x,a)\big)\\
&\leq   \omega_f(\lambda\cdot d(x,A)+\gamma +d(x,a)),
\end{align*}\hfill
\endgroup

\noindent  and we are done.
\end{proof}

We can now prove the main theorem of this section.

\begin{proof}[Proof of Theorem \ref{Thmcontcoarproj}]
Let $\theta:M\to E$ be the continuous extension of $\varphi_{|A}$  given by  Proposition \ref{prop50} for $\lambda=2$, and some $\gamma>0$ such that 

$$\omega_\varphi(\gamma)+\omega_{\varphi_{|A}}(4\gamma)\leq \inf_{s>0}\omega_\varphi(s)+\inf_{s>0}\omega_{\varphi_{|A}}(s)+\delta.$$\hfill

 Let $U=\{x\in M\mid d(x,A)< \gamma\}$, and let $\mathcal{U}=\{U_j\}_{j\in J}$ be an open cover for the metric space $M\setminus A$ such that $\text{diam}(U_j)<\gamma$, for all $j\in J$. So, $\mathcal{U}'=\{U,U_j\}_{j\in J}$ is an open cover for $M$, and, as $M$ is paracompact, $\mathcal{U}'$ has a locally finite refinement (see \cite{M}, Theorem 41.4). Hence, there exists a family of open sets $\mathcal{V}=\{V_i\}_{i\in I}$ refining $\mathcal{U}$ such that $\{U,V_i\}_{i\in I}$ is a locally finite open cover of $M$.  For each $i\in I$, pick $x_i\in V_i$,  let $\psi_i(x)=d(x_i,V^c_i)$,  and let $\psi_U(x)=\max\{0,1-d(x,A)/\gamma\}$, for all $x\in M$. So $\psi_U(x)=1$, if $x\in A$, and $\psi_U(x)=0$, if $x\not\in U$.

Let $\Psi=\psi_U+\sum_{i\in I}\psi_i$, and define $\Phi: M\to  \text{conv}(\varphi(M))$ by

$$\Phi(x)=\frac{\psi_U(x)}{\Psi(x)}\theta(x)+\sum_{i\in I}  \frac{\psi_i(x)}{\Psi(x)}\varphi(x_i).
$$\hfill

\noindent As $\{U,V_i\}_{i\in I}$ is  locally finite, $\Phi$ is continuous. Also, as $\psi_i(x)=0$, for all $x\in A$, and all $i\in I$, it is clear that $\Phi_{|A}=\varphi_{|A}$. 

Let $x\in M\setminus A$, and let $V_{i_1},\ldots ,V_{i_k}$ be the only elements of $\mathcal{V}$ containing $x$.  As $\text{diam}(V_i)< \gamma$, for all $i\in I$, we have that $d(x,x_{i_j})<\gamma$, for all $j\in\{1,\ldots ,k\}$. Hence, 

\begin{align*}
\|\varphi(x)-\Phi(x)\|&\leq \frac{\psi_U(x)}{\Psi(x)}\cdot \|\varphi(x)-\theta(x)\|+\sum_{j=1}^k\frac{\psi_{i_j}(x)}{\Psi(x)}\cdot\|\varphi(x)-\varphi(x_{i_j})\|\\
&\leq \frac{\psi_U(x)}{\Psi(x)}\cdot \|\varphi(x)-\theta(x)\|+ \omega_\varphi(\gamma)\cdot \sum_{j=1}^k\frac{\psi_{i_j}(x)}{\Psi(x)}.
\end{align*}

\noindent If $x\not\in U$, this shows that $\|\varphi(x)-\Phi(x)\|\leq  \omega_\varphi(\gamma)$. If $x\in U$, pick $a\in A$ such that $d(x,a)<\gamma$. Then, as $\theta(a)=\varphi(a)$, we have that

\begin{align*}
\|\varphi(x)-\theta(x)\|&\leq \|\varphi(x)-\varphi(a)\|+\|\theta(a)-\theta(x)\|\\
&\leq \omega_\varphi(\gamma)+\omega_{\varphi_{|A}}(\lambda\cdot d(x,A)+d(x,a)+\gamma)\\
&\leq \omega_\varphi(\gamma)+\omega_{\varphi_{|A}}(4\gamma).
\end{align*}\hfill

\noindent So, we are done.
\end{proof}

\begin{cor}\label{contcoarproj2}
Let $Y$ be a Banach space and $A\subset Y$ be a closed subset. Let  $\varphi:Y\to A$ be a  retraction. Then, for all $\delta>0$, there exists a continuous retraction $\Phi:Y\to \text{conv} (A)$ such that

$$\sup_{x\in Y}\|\varphi(x)-\Phi(x)\|\leq \inf_{s>0}\omega_\varphi(s)+\delta.$$\hfill

\noindent In particular, if $\varphi$ is coarse, so is $\Phi$.
\end{cor}

\begin{proof}
As $\varphi_{|A}=\text{Id}_{|A}$, we have that $\omega_{\varphi_{|X}}(t)=t$, for all $t$. A straightforward  application of Theorem \ref{Thmcontcoarproj} finishes the proof. 
\end{proof}

\noindent \emph{Proof of Corollary \ref{contcoarproj}.} This is a particular case of Corollary \ref{contcoarproj2} above.\qed\\

In the case where $A=\emptyset$, the $\Phi$ given by Theorem \ref{Thmcontcoarproj} is not only continuous, but even locally Lipschitz. Let $(M,d)$ and $(N,\partial)$ be metric spaces. We call a map $f: M\to N$ \emph{locally Lipschitz} if for each $x\in M$, there exists a neighborhood of $x$ in which $f$ is Lipschitz. 
 
\begin{prop}\label{coarsecont}
Let $(M,d)$ be a metric space, and let $E$ be a normed space. Let $\varphi:M\to E$ be a  map. Then, for all $\delta>0$, there exists a locally Lipschitz map $\Phi:M\to \text{conv}(\varphi(M))$ such that

$$\sup_{x\in M} \|\varphi(x)-\Phi(x)\|\leq \inf_{s>0}\omega_\varphi(s)+\delta.$$\hfill

\noindent In particular, if $M$ coarsely embeds into $E$, then $M$   coarsely embeds into $E$ by a locally Lipschitz map.
\end{prop}		

\begin{proof}
Let $\gamma$, $\mathcal{V}=\{V_i\}_{i\in I}$, $\{x_i\}_{i\in I}$, $(\psi_i)_{i\in\N}$, $\Psi$ and $\Phi$ be as in  the proof of Theorem \ref{Thmcontcoarproj} (with $A=\emptyset$,  and $U=\emptyset$). We only need to notice that $\Phi$ is locally Lipschitz. Let $x\in M$. Then, there exists $\eps>0$ such that $B(x,\eps)$ intersects only finitely many elements of $\mathcal{V}$, say $V_{i_1},\ldots ,V_{i_k}$. Without loss of generality, we can assume that $x\in V_{i_1}$, and that $B(x,2\eps)\subset V_{i_1}$. So $\Psi(y)\geq\eps$, for all $y\in B(x,\eps)$. Therefore, as $\psi_i(y)/\Psi(y)\leq 1$, for all $y\in M$, and all $i\in I$, we have that

\begin{align*}
\Big|\frac{\psi_i(z)}{\Psi(z)}-\frac{\psi_i(y)}{\Psi(y)}\Big|&\leq \frac{|\psi_i(z)-\psi_i(y)|}{\Psi(z)}+\frac{|\Psi(z)-\Psi(y)|}{\Psi(z)}\cdot \frac{\psi_i(y)}{\Psi(y)}\\
&\leq \Big(\frac{1+k}{\eps}\Big)d(z,y),
\end{align*}\hfill

\noindent for all $z,y\in B(x,\eps)$. Hence, letting $L= \max\{\|\varphi(x_{i_l})\|\mid 1\leq l\leq k\}$, we have

$$\|\Phi(z)-\Phi(y)\|\leq L\Big(\frac{k+k^2}{\eps}\Big)d(z,y),$$\hfill

\noindent for all $z,y\in B(x,\eps)$. 
\end{proof}

We had just shown that if $(M,d)$ coarsely embeds into a Banach space $E$, then it coarsely embeds by a continuous map. We would like to obtain that the existence of coarse embeddings actually guarantee us the existence of simultaneously coarse and homeomorphic embeddings. In the next proposition, we show that injectivity of the embedding is not a problem.

\begin{prop}
Let $(M,d)$ be a separable metric space and let $E$ be an infinite dimensional Banach space. Let  $\varphi:M\to E$ be a  map. Then, for all $\delta>0$, there exists an injective continuous  map $\Phi:M\to E$ such that

$$\sup_{x\in M}\|\varphi(x)-\Phi(x)\|\leq \inf_{s>0}\omega_\varphi(s)+\delta.$$\hfill

\noindent In particular, if a separable Banach space $X$ coarsely embeds into a Banach space $Y$, then $X$ coarsely embeds into $Y$ by an injective continuous map.
\end{prop}

\begin{proof}
Let $\varphi: M\to E$ be a coarse map, and $\delta>0$. Pick $\gamma>0$ such that $\omega_\varphi(\gamma)+2\gamma<\inf_{s>0}\omega_\varphi(s)+\delta$. Let $Z\subset E$ be a closed infinite dimensional separable subspace such that the quotient space $E/Z$ is infinite dimensional. As $M$ is separable, $M$ isometrically embeds into the space of continuous function on $[0,1]$ with the supremum norm, $C[0,1]$ (see \cite{FHHMZ}, Corollary 5.9). Therefore, as $C[0,1]$ is homeomorphic to $B_Z$ (see \cite{K}), it follows that $M$ homeomorphically embeds into $\gamma\cdot B_Z$. Say $\theta:M\to \gamma\cdot B_Z$ is such embedding.	

Let $\mathcal{U}=\{U_n\}_{n\in \N}$ be a countable locally finite cover of $M$ such that $\text{diam}(U_n)<\gamma$, for all $n\in\N$. For each $n\in\N$, pick $x_n\in U_n$, and let $\psi_n(x)=d(x,U^c_n)$, for all $x\in M$.
 
Define a sequence $(y_n)_n$ in $E$ as follows. Pick $y_1\in B(\varphi(x_1),\gamma)\setminus Z$. Say $y_1,\ldots,y_k$ had been chosen. Then pick $y_{k+1}\in B(\varphi(x_{k+1}),\gamma)\setminus(Z\oplus\text{span}\{y_1,\ldots,y_k\})$. Let $\Psi=\sum_{n\in\N}\psi_n$, and define $\Phi:M\to E$ by

$$\Phi(x)=\theta(x)+\sum_{n\in\N}  \frac{\psi_n(x)}{\Psi(x)}y_n.
$$\hfill

\noindent for all $x\in M$. Clearly, $\Phi$ is continuous, and satisfies the required inequality. To notice that $\Phi$ is injective, notice that, by our choice of $(y_n)_n$, if $\Phi(x)=\Phi(y)$, then $\psi_n(x)/\Psi(x)=\psi_n(y)/\Psi(y)$, for all $n\in\N$. So, $\theta(x)=\theta(y)$, which implies $x=y$.

The last claim follows from the facts that (i) if $\text{dim}(X)<\infty$, then $\text{dim}(Y)\geq \text{dim}(X)$ (see \cite{NY}, Theorem 2.2.5 and Example 2.2.6), and (ii) if an infinite dimensional Banach space $X$  coarsely embeds into $Y$, then $Y$ is also infinite dimensional.
\end{proof}

\section{Simultaneously homeomorphic and coarse embeddings.}\label{sectioncoarsecont}

In this section, we show that if a Banach space $X$ coarsely embeds into a minimal Banach space $Y$, then $X$ simultaneously homeomorphically and coarsely embeds into $Y$. In order to show that, we show that  there exists a map $X\to (\oplus Y)_\mathcal{E}$, where $\mathcal{E}$ is any $1$-unconditional basic sequence, which is simultaneously a homeomorphic and  coarse embedding. \\

The following lemma is an application of the methods of \cite{Ro} to our specific setting (see \cite{Ro}, Lemma 16).

\begin{lemma}\label{christian2}
Suppose $X$ and $E$ are Banach spaces and $P_n : E\to E$ is a sequence of bounded projections onto subspaces $E_n\subset E$ so that $E_m \subset \text{Ker} (P_n)$, for all $m\neq n $. Assume also that, for all $n\in\N$, there exists a coarse embedding $\sigma_n : X\to E_n$ which is also continuous. Then $X$  homeomorphically coarsely embeds into $E$ by a map with uniformly continuous inverse.
\end{lemma}

\begin{proof}
Let us define a continuous coarse map $\psi:X\to \text{Ker} (P_1)$ such that $\psi^{-1}$ exists and is uniformly continuous. Then, by setting $\Psi:X\to E$ as $\Psi(x)=\sigma_1(x)+\psi(x)$, for all $x\in X$, we have that $\Psi$ is a continuous coarse embedding with uniformly continuous inverse. Indeed,  $\Psi$ is clearly coarse and continuous.  As $\|\sigma_1(x)-\sigma_1(y)\|=\|P_1(\Psi(x)-\Psi(y) )\|\leq \|P_1\|\cdot\|\Psi(x)-\Psi(y)\|$, for all $x,y\in X$, it follows that $\Psi$ is expanding. As $\|\psi(x)-\psi(y)\|=\|(\text{Id}-P_1)(\Psi(x)-\Psi(y) )\|\leq \|\text{Id}-P_1\|\cdot\|\Psi(x)-\Psi(y)\|$, for all $x,y\in X$, it follows that $\Psi$ has uniformly continuous inverse. 
  
Without loss of generality, we can assume that $\sigma_n(0)=0$, for all $n\in\N$. As each $\sigma_n$ is a coarse embedding, there exist sequences $(L_n)_{n\in\N}$ and $(\Delta_n)_{n\in\N}$ of positive numbers such that $\omega_{\sigma_n}(t)\leq L_nt+L_n$ (see Proposition \ref{KALON1}) and $\rho_{\sigma_n}(\Delta_n)>1$,  for all $n\in\N$, and all $t\in [0,\infty)$. We can assume that $\Delta_n\geq 1$, for all $n\in\N$. For each $n\in\N$, let $\psi_n:X\to E_n$ be given by

$$\psi_n(x)=\frac{\sigma_n(n\Delta_nx)}{n\Delta_nL_n2^n},$$\hfill

\noindent   and let $\psi(x)=\sum_{n>1}\psi_n(x)$, for all $x\in X$. Clearly, $\psi_n(0)=0$, for all $n\in\N$, and $\psi(0)=0$.\\

\textbf{Claim:} $\psi$ is well defined, coarse, continuous, and $\psi^{-1}$ is uniformly continuous.\\

 For all $x,y\in X$, and all $n\in\N$, there are $x_0,\ldots ,x_n\in X $, such that $x_0=n\Delta_nx$, $x_n=n\Delta_ny$, and $\|x_{j-1}-x_{j}\|=\Delta_n\|x-y\|$, for all $1\leq j\leq n$. So, by the triangle inequality,

$$\|\sigma_n(n\Delta_nx)-\sigma_n(n\Delta_ny)\|\leq \sum_{j=1}^n \|\sigma_n(x_{j-1})-\sigma_n(x_j)\|\leq n\cdot \omega_{\sigma_n}(\Delta_n\|x-y\|).$$\hfill

\noindent Hence, as $\Delta_n\geq 1$, for all $n\in\N$, we have that

\begin{align*}
\|\sum_{n=l}^m\psi_n(x)-\sum_{n=l}^m\psi_n(y)\|&\leq \sum_{n=l}^m\frac{\|\sigma_n(n\Delta_nx)-\sigma_n(n\Delta_ny)\|}{n\Delta_nL_n2^n}\\
&\leq \sum_{n=l}^m\frac{\omega_{\sigma_n}(\Delta_n\|x-y\|)}{\Delta_n L_n2^n}\leq\frac{\|x-y\|+1}{2^{l-1}}.
\end{align*}\hfill
 
In particular, as $\psi_n(0)=0$, for all $n\in\N$, we have that $\|\sum_{n=l}^m\psi_n(x)\|\leq(\|x\|+1)/2^{l-1}$, for all $x\in X$, and all $l,m\in\N$, with $l\leq m$.  Hence,  $\psi$ is well defined. Similarly, the argument above gives us that $\omega_\psi(t)\leq t+1$, for all $t>0$, so $\psi$ is coarse. 

Let $x\in X$, and $\eps>0$. Choosing $N\in \N$ such that $1/2^N<\eps/4$, we have that, for all $y\in X$, with $\|x-y\|\leq 1$,

\begin{align*}
\|\psi(x)-\psi(y)\|&\leq \sum_{n\leq N}\|\psi_n(x)-\psi_n(y)\|+ \sum_{n>N}\frac{\|x-y\|+1}{2^n}\\
&\leq  \sum_{n\leq N}\|\psi_n(x)-\psi_n(y)\|+\frac{\eps}{2}.
\end{align*}\hfill

\noindent By the continuity of each $\psi_n$ at $x$, there exists $\delta\in(0,1)$ such that $\sum_{n\leq N}\|\psi_n(x)-\psi_n(y)\|<\eps/2$, whenever $\|x-y\|<\delta$. Then, $\|\psi(x)-\psi(y)\|<\eps$, if $\|x-y\|<\delta$. So $\psi$ is continuous.

Let us show that $\psi^{-1}$ exists and it is uniformly continuous. For this, we only need to show that, for all $\eps>0$, there exists $\delta>0$ such that, for all $x,y\in X$,

$$\|x-y\|>\eps \ \Rightarrow \  \|\psi(x)-\psi(y)\|>\delta.$$\hfill 

As $\rho_{\sigma_n}(\Delta_n)>1$, for all $n\in\N$, if $x,y\in X$ and $\|x-y\|>1$, then   $\|\sigma_n(\Delta_nx)-\sigma_n(\Delta_ny)\|>1$. Fix $\eps>0$, and pick  $n\in\N$ such that $1/n<\eps$. Then, if $\|x-y\|>\eps$, we have that

\begin{align*}
\|\psi(x)-\psi(y)\|&\geq \frac{\|\sigma_n(n\Delta_nx)-\sigma_n(n\Delta_ny)\|}{\|P_n\|n\Delta_nL_n2^n}\\
&\geq \frac{1}{\|P_n\|n\Delta_nL_n2^n},
\end{align*}\hfill

\noindent  Hence, $\psi^{-1}$ is uniformly continuous, and we are done.\end{proof}

\noindent \emph{Proof of Theorem \ref{mainmain}(ii).} Let $\varphi: X\to Y$ be a coarse embedding. By Theorem \ref{Thmcontcoarproj}, we can assume that $\varphi$ is also continuous. As in the proof of item (i) of Theorem \ref{mainmain},  $Y$  contains an unconditional basic sequence $(e_n)_n$. Let $(A_n)_n$ be a partition of $\N$ into infinite subsets, and set $E=\overline{\text{span}}\{e_j\mid j\in\N\}$ and  $E_n=\overline{\text{span}}\{e_j\mid j\in A_n\}$, for all $n\in\N$. As $Y$ is minimal, there exists a sequence of isomorphic embeddings $T_n:Y\to E_n$. So, $T_n\circ\varphi$ is a continuous coarse embedding of $X$ into $E_n$, for all $n\in\N$.  For each $n\in\N$, let $P_n:E\to E_n$ denote the natural projection.  We can now apply Lemma \ref{christian2}, so, $X$ simultaneously homeomorphically and coarsely embeds into $Y$ by a map with uniformly continuous inverse. 
\qed\\

The following corollary is a trivial consequence of Theorem \ref{mainmain}(ii).

\begin{cor}\label{cor2222}
If a Banach space $X$ coarsely embeds into $T^*$ (resp. $S$), then $X$ simultaneously homeomorphically and coarsely embeds into $T^*$ (resp. $S$) by a map with uniformly continuous inverse.
\end{cor}

\noindent \emph{Proof of Theorem \ref{homeocoarseemb}.}
If $X$ coarsely embeds into $Y$, by Theorem \ref{Thmcontcoarproj}, $X$ coarsely embeds into $Y$ by a continuous map. Let $E=(\oplus Y)_\mathcal{E}$, and $E_n=\{(x_n)_n\in E\mid \forall j\neq n,\  x_j=0\}$, for all $n\in\N$. Then, by  Lemma \ref{christian2}, $X$ homeomorphically coarsely embeds into $E$ by a map with uniformly continuous inverse.
\qed\\

The following simpler version of Problem \ref{problemaprin} could be slightly easier to be proven, and it would be a significant advance on this problem. 

\begin{problem}
Let $X$ and $Y$ be Banach spaces, and assume that   $X$ coarsely embeds into $Y$. Does $B_X$ uniformly embed into $Y$? What if $Y$ is minimal?
\end{problem}

It is worth noticing that one cannot hope that $X$ coarsely embeds into $Y$ if and only if $B_X$ uniformly embeds into $Y$ (even if we restrict ourselves to minimal spaces $Y$). Indeed, it is well known that all the $\ell_p$'s have uniformly homeomorphic balls (see \cite{OS}, Theorem 2.1), but $\ell_p$ does not coarsely embed into $\ell_2$ for any $p>2$ (see \cite{JR}, Theorem 1, or \cite{MN}, Theorem 1.11).

\section{Continuous coarse sections.}\label{sectioncoarsehomeo}

In \cite{Ka}, Kalton proved (Theorem $8.9$) that the concepts of coarse equivalence and uniform homeomorphism are actually distinct concepts, i.e.,  Kalton presented two Banach spaces $X$ and $Y$ which are coarsely equivalent but not uniformly homeomorphic. However, the coarse equivalence presented in \cite{Ka} only preserves the large scale geometries of $X$ and $Y$ and does not need to be a homeomorphism. In this section, we show that Kalton's example is actually an example of Banach spaces which are simultaneously  homeomorphically and coarsely equivalent, but  not uniformly homeomorphic. 

Let $X$ and $Y$ be Banach spaces, and let $Q:Y\to X$ be a quotient map. If $A\subset X$, we say that $f:A\to Y$ is a \emph{section of $Q$} if $Q\circ f=\text{Id}_A$. Kalton's argument is based on the construction of a quotient map $Q:Y\to X$ for which a coarse section $X\to Y$ exists, but $X$ does not coarse Lipschitz embed into $Y$ by map which is also uniformly continuous (see \cite{Ka}, Theorem 8.8). In particular, $Q$ has no uniformly continuous section $X\to Y$. In this section, we show that if a quotient map $Q: Y\to X$ admits a coarse section, then it admits a continuous coarse section. As a corollary, we get the strengthening of Kalton's result mentioned above.\\

The proof of the following lemma uses ideas in the proof of Proposition 6.5 of \cite{Ka}.

\begin{lemma}\label{contsec}
Let $X$ and $Y$ be Banach spaces, and let $Q:Y\to X$ be a quotient map.  Assume that there exists a coarse section $\varphi: X\to Y$. Then, there exists $L>1$ such that, for every $\eps>0$, there exists a continuous section $\psi:\partial B_X\to Y$ of cL-type $(L,\eps)$.
\end{lemma}

\begin{proof}
Let $\varphi:X\to Y$ be a coarse section. So, there exists $L>1$ such that $\omega_\varphi(t)\leq Lt+L$, for all $t>0$ (see Proposition \ref{KALON1}). Fix $\eps\in (0,1)$, and let us show that the required continuous section $\psi$ of cL-type $(L,\eps)$ exists. 

For each $n\in\N$, let $\varphi_n(x)=\varphi(nx)/n$. So,  each $\varphi_n$ is a coarse section, and $\omega_{\varphi_n}(t)\leq Lt+Ln^{-1}$, for all $t>0$. For each $n\in\N$, let $\Phi_n:X\to Y$ be the continuous map given by Theorem \ref{Thmcontcoarproj} applied to $\varphi_n$, $L/n$, and $A=\emptyset$. Hence, we have that 

$$\sup_{x\in X} \|\varphi_n(x)-\Phi_n(x)\|\leq \frac{2L}{n},$$\hfill

\noindent for all $n\in\N$. In particular, $\omega_{\Phi_n}(t)\leq Lt+ 5L/n$, and 

$$\|x-Q(\Phi_n(x))\|\leq \|x-Q(\varphi_n(x))\|+\|Q(\varphi_n(x))-Q(\Phi_n(x))\|\leq \frac{2L\|Q\|}{ n}$$\hfill

\noindent for all $n\in\N$, and all $x\in X$. Let $\lambda\in (0,1)$ be such that $\sum_{n\in\N}\lambda^n< \frac{\eps}{8L}$. Fix $n_0\in\N$ large enough so that $2L\|Q\|/n_0<\lambda$, and $5L/n_0<\eps/2$.

Let  $h:X\to Y$ be given by

\begin{align*}
h(x)=\left\{\begin{array}{ll}
0, & \text{ if }\ \ x=0,\\
\frac{\|x\|}{2}\Big(\Phi_{n_0}\Big(\frac{x}{\|x\|}\Big)-\Phi_{n_0}\Big(-\frac{x}{\|x\|}\Big)\Big), & \text{ if }\ \ x\neq 0.\end{array}\right.
\end{align*}\hfill

\noindent Then  $h$ is continuous, positively homogeneous, and bounded on bounded sets. Also, it is clear that $\|x-Q(h(x))\|\leq \lambda\|x\|$, for all  $x\in X$. Let $g(x)=x-Q(h(x))$. Then, as $g$ is positively homogeneous, we have that $\|g^n(x)\|\leq \lambda^n \|x\|$, for all $n\in\N$, and all $x\in X$. Set $g^0(x)=x$  and let

$$\psi(x)=\sum_{n=0}^\infty h(g^n(x)),$$\hfill

\noindent for all $x\in \partial B_X$.

As $h$ is positively homogeneous, the series above converges uniformly on bounded sets, so $\psi$ is well defined and continuous.  Also, as $g^n(x)-Q(h(g^n(x)))=g^{n+1}(x)$, we have that

$$Q(\psi(x))=\sum_{n=0}^\infty (g^n(x)-g^{n+1}(x))=x,$$\hfill

\noindent for all $x\in X$. So $\psi:\partial B_X\to Y$ is a continuous section of $Q$.

It remains to notice that $\psi$ is of cL-type $(L,\eps)$. Notice that, as $5L/n_0<\eps/2$, we have that $\sup_{x\in B_X}\|h(x)\|< 2L$ and $\omega_{h_{|\partial B_X}}(t)\leq Lt+\eps/2$, for all $t>0$. Then

\begin{align*}
\|\psi(x)-\psi(y)\|&\leq \|h(x)- h(y)\|+\sum_{n=1}^\infty \|h(g^n(x))\|+\sum_{n=1}^\infty \|h(g^n(y))\|\\
&\leq L\|x-y\|+\frac{\eps}{2}+ 2\cdot  \sup_{x\in B_X}\|h(x) \|\cdot\sum_{n=1}^\infty\lambda^n\\
&\leq  L\|x-y\|+\eps,
\end{align*}\hfill

\noindent for all $x,y\in \partial B_X$. So, $\psi$ is of cL-type $(L,\eps)$, and we are done.
\end{proof}

The next technical  lemma is the continuous version of Lemma $7.4$ of \cite{Ka}, and it will play a fundamental role in the proof of Theorem \ref{homeocoarseequiv}.

\begin{lemma}\label{teckal}
Let $X$ and $Y$ be Banach spaces and consider a map $t\in[0,
\infty)\mapsto f_t\in \mathcal{H}(X,Y)$ with the property that, for some $K>0$,

$$\|f_t\|_{e^{-2t}}\leq K,\ \text{ and }\ \|f_t-f_s\|\leq K|t-s|, \ \ \forall \ t,s\geq 0.$$\hfill

\noindent Define $F:X\to Y$ as 

\begin{align*}F(x)=\left\{\begin{array}{ll}
%0, & x=0,\\
f_0(x), & \|x\|\leq 1,\\
f_{\ln\|x\|}(x), & \|x\|>1.
\end{array}\right.\end{align*}\hfill

\noindent Then $F$ is coarse. Moreover, if $f_t\in\mathcal{HC}(X,Y)$, for all $t\geq 0$, then $F$ is continuous.
\end{lemma}

In Lemma $7.4$ of \cite{Ka}, the author shows that the map $F$ above is coarse, and, under the assumption that $f_t$ is \emph{uniformly} continuous, for all $t\geq 0$, the author shows that $F$ is also uniformly continuous. Therefore, we only present the proof that $F$ is continuous if each $f_t$ is so.\smallskip

\noindent \textit{Sketch of the proof.} For convenience, let $f_t=f_0$, if $t<0$. In the proof of Lemma 7.4 of \cite{Ka}, Kalton shows that 

\begin{align}\label{equa}\|F(x)-F(z)\|\leq 3K\|x-z\|+2K\min\{\|x\|,\|x\|^{-1}\}+2K\min\{\|z\|,\|z\|^{-1}\},
\end{align}\hfill

\noindent for all $x,z\in X$. In particular, $\omega_F(t)\leq 3Kt+4K$, so $F$ is coarse. 

Let us show that $F$ is continuous if each $f_t\in \mathcal{HC}(X,Y)$, and  the map $t\mapsto f_t$ is continuous.  Note that, as $F(x)=f_0(x)$ if $\|x\|\in[0,1)$, $F$ is continuous at $x$ if $\|x\|\in [0,1)$.  Therefore, we only need to show that $F$ is continuous at $x$ if $\|x\|\geq 1$. 

Let $x\in X$, with $\|x\|\geq 1$, and fix $\eps>0$. Pick $\delta_0\in (0,1)$ such that $K\delta_0<\eps/6$, and $a>1$ such that $4K/a<\eps/2$.  If $\|x\|>a$, pick $\delta_1\in (0, \min\{\delta_0,\|x\|-a\})$. By Equation \ref{equa}, if $\|x-z\|<\delta_1$, we have

$$\|F(x)-F(z)\|\leq 3K\delta_1+4Ka^{-1}<\eps.$$\hfill

Say $\|x\|\leq a$. Let $b=\ln(a+1)$. Pick $N>b$ such that $Kb/N<\eps/(3e^b)$. Then $|s-t|\leq b/N$  implies $\|f_s-f_t\|<\eps/(3e^b)$. 

By the continuity of each $f_t$, there exists $\delta_2\in (0,\min\{b/N,1\})$ such that $\|x-z\|<\delta_2$ implies

$$\|f_{kb/N}(x)-f_{bk/N}(z)\|<\eps/3, \ \ \text{ for all }\ \  0\leq k\leq N.$$\hfill

\noindent Making $\delta_2$ smaller if necessary, we can also assume that $\|x-z\|<\delta_2$ implies $|\ln\|x\|-\ln\|z\||<b/(2N)$.

Fiz $z\in X$ with $\|x-z\|<\delta_2$. As $\|x\|\in[ 1, a]$,  we have that $\ln\|x\|,\max\{\ln\|z\|,0\}\in [0, b]$. Therefore, as $|\ln\|x\|-\ln\|z\||<b/(2N)$, there exists $k\in\{0,\ldots ,N\}$ such that

$$\left|\ln\|x\|-\frac{kb}{N}\right|,\left|\max\{\ln\|z\|,0\}-\frac{kb}{N}\right|\leq \frac{b}{N}.$$\hfill

\noindent As $\|x-z\|<\delta_2$,  we have that $\|x\|/e^b,\|z\|/e^b\leq 1$. Therefore, we conclude that

\begin{align*}\|F(x)-F(z)\|&\leq \|F(x)-f_{kb/N}(x)\|+\|f_{kb/N}(x)-f_{kb/N}(z)\|\\
&\ \ \ \ +\|f_{kb/N}(z)-F(z)\|\\
&\leq \frac{\eps\|x\|}{3e^b}+ \frac{\eps}{3}+\frac{\eps\|z\|}{3e^b}\leq \eps.
\end{align*}\hfill

\noindent So, $F$ is continuous at $x$.\qed
\smallskip

\begin{prop}\label{existcontsec}
Let $Q:Y\to X$ be a quotient map. Assume that there exists a constant $L>1$, a sequence $(\eps_n)_n$ of positive real numbers converging to zero, and a sequence of continuous sections $\varphi_n:\partial B_X\to Y$ such that $\varphi_n$ is of $cL$-type $(L,\eps_n)$, for all $n\in\N$. Then $Q$ has a continuous coarse section.
\end{prop}

\begin{proof}
Without loss of generality, we may assume that $\eps_n<e^{-2n}$, for all $n\in\N$. For each $n\in\N$, let $\psi_n(x)=1/2(\varphi_n(x)-\varphi_n(-x))$, for all $x\in \partial B_X$. So, each $\psi_n$ is a continuous section of $Q$ of cL-type $(L,\eps_n)$ and  $\|\psi_n(x)\|\leq 2L$, for all $n\in\N$, and all $x\in \partial B_X$.

By Proposition \ref{estim}, we can extend each $\psi_n$ to an $f_{n-1}\in\mathcal{HC}(X,Y)$ so that $f_{n-1}$ is a section of $Q$, and $\|f_{n-1}\|_{e^{-2n}}\leq 8L$, for all $n\in\N$.  For each $t\geq 0$, we define $f_t:X\to Y$ as follows. If $t\in[n-1,n]$, let

$$f_t(x)=(n-t)f_{n-1}(x)+(t-n+1)f_{n}(x).$$\hfill

\noindent Clearly $t\mapsto f_t$ is  continuous. Indeed, $\|f_t-f_s\|\leq 4L|t-s|$, for all $t,s\in [n-1,n]$.  

Notice that $\|f_t\|_{e^{-2t}}\leq 8L$, for all $t\geq 0$. Let $F$ be the map obtained by Lemma \ref{teckal} for the maps $(f_t)_{t\geq 0}$. Then $F$ is a continuous  coarse section of $Q$.
\end{proof}

\noindent \emph{Proof of Theorem \ref{homeocoarseequiv}.}
If the quotient map $Q:Y\to X$ has a coarse section $X\to Y$, it follows from Lemma \ref{contsec} and Proposition \ref{existcontsec} that $Q$ has a continuous coarse section. Let $\varphi:X\to Y$ be such section. Then, the map $y\mapsto (y-\varphi(Q(y)),Q(y))$ is both a homeomorphism and a  coarse equivalence between  $Y$  and $\text{Ker}(Q)\oplus X$ with inverse $(x,z)\mapsto x+\varphi(z)$.
\qed\\

\noindent \emph{Proof of Corollary \ref{homeocoarseequivC}.}
By Proposition 8.4 and Theorem 8.8 of \cite{Ka}, there exist separable Banach spaces $X$ and $Y$, and a quotient map $Q:Y\to X$ which admits a coarse section $\varphi: X\to Y$, but $X$ does not coarse Lipschitz embed into $Y$ by a uniformly continuous map. Hence,  $Y$ and $\text{Ker}(Q)\oplus X$ are not uniformly homeomorphic.  By Theorem \ref{homeocoarseequiv},  $Y$ and $\text{Ker}(Q)\oplus X$ are simultaneously homeomorphically and coarsely equivalent.
\qed\\

This raises the question of when two Banach spaces are simultaneously homeomophically and coarsely equivalent. It is well known that any two Banach spaces with the same density character are homeomorphic (see \cite{K}, and \cite{T}). But what about if $X$ and $Y$ are coarsely equivalent? Can we get both coarse equivalence and topological equivalence at the same time?

\begin{problem}
Let $X$ and $Y$ be Banach spaces, and assume that $X$ and $Y$ are coarsely equivalent. Are $X$ and $Y$ simultaneously coarsely and homeomorphically equivalent?
\end{problem}

It is worth noticing that, for separable spaces, the existence of a  coarse equivalence easily implies the existence of a measurable coarse equivalence.

\begin{prop}\label{measurable}
Let $X$ and $Y$ be separable Banach spaces, and assume that $X$ is coarsely equivalent to $Y$. Then, there exists a coarse equivalence $X\to Y$ which is also a Borel bijection.   
\end{prop}

\begin{proof}
Without loss of generality, we can assume that $X$ and $Y$ are infinite dimensional (see Proposition 2.2.4, and Theorem 2.2.5 of \cite{NY}). Let $\{x_n\}_n$ and $\{y_n\}_n$ be  $(1,1)$-nets in $X$ and $Y$ such that $x_n\mapsto y_n$ defines a Lipschitz isomorphism. Let $A_1=B(x_1,1)\setminus \cup_{i>1}B(x_i,1/2)$, and 

$$A_n=B(x_n,1)\setminus \big(\bigcup_{i<n}A_i\cup\bigcup_{i>n}B(x_i,1/2)\big),$$\hfill

\noindent for all $n>1$. We define a sequence of subsets $(C_n)_n$ of $Y$ analogously. It is clear that $X=\sqcup_n A_n$, $Y=\sqcup_n C_n$, that $A_n$ and $C_n$ are Borel, and that $A_n$ and $C_n$ are Borel isomorphic (see \cite{Ke}, Theorem 15.6), for all $n\in\N$. Let $f_n:A_n\to C_n$  be  Borel isomorphisms. Define a map $\varphi: X\to Y$ by setting $\varphi(x)=    f_n(x)$, if $x\in A_n$. It should be clear that $\varphi$ is both a coarse equivalence and a Borel bijection.
\end{proof}

\section{Unconditional sums of coarse and uniform equivalences.}\label{sectionsums}

 In \cite{Ka2}, Kalton proved  (Theorem $4.6$(ii)) that if $X$ and $Y$ are coarsely equivalent (resp. uniformly homeomorphic), then $\ell_p(X)$ and $\ell_p(Y)$ are coarsely equivalent (resp. uniformly homeomorphic). However, as Kalton pointed out, his proof seems to be much more complicated than necessary, and it relies on results about  close (resp. uniformly close) Banach spaces. In this section, we give a direct proof for  a general theorem (see Theorem \ref{geral} below) which gives us Kalton's result as a corollary. \\

\noindent \emph{Proof of Theorem \ref{trikal}.}  Let $\varphi: X\to Y$ be a coarse equivalence (resp. uniform homeomorphism). Assume $\varphi(0)=0$. For each $n\in\N$, let $\varphi_n(\cdot)=2^{-n}\varphi(2^n\cdot)$. Define $\Phi:(\oplus X)_\mathcal{E}\to (\oplus Y)_\mathcal{E}$ by letting $\Phi(x)=(\varphi_n(x_n))_n$, for all $x=(x_n)_n\in (\oplus X)_\mathcal{E}$. \\

\textbf{Claim:} $\Phi$ is  well defined and coarse (resp. uniformly continuous).\\
 
Let  $L>0$, be such that $\omega_\varphi(t)<Lt+L$, for all $t>0$. So, $\omega_{\varphi_n}(t)<Lt+L2^{-n}$, for all $t>0$. Let us first notice that $\Phi$ is well defined. Let $x=(x_n)_n\in (\oplus X)_\mathcal{E}$. For $\eps>0$, pick $N\in\N$ so that $\|\sum_{n>N} \|x_n\| e_n\|<\eps /2L$, and $\sum_{n>N}2^{-n}<\eps /2L$. Then, for $k>l>N$, we have

\begin{align*}\|\sum_{n=l}^k\|\varphi_n(x_n)\|e_n\|&\leq \|\sum_{n=l}^k\Big(L\|x_n\|+\frac{L}{2^n}\Big)e_n\|\\
&\leq \|\sum_{n=l}^kL\|x_n\|e_n\|+\|\sum_{n=l}^k\frac{L}{2^n}e_n\| <\eps.
\end{align*}\hfill

\noindent Hence, the sum $\sum_{n\in\N}\varphi_n(x)$ converges for every $x$, so $\Phi$ is well defined. 

Say $x,y\in (\oplus X)_\mathcal{E}$. Then

\begin{align*}\|\Phi(x)-\Phi(y)\|&=\|\sum_{n\in\N}\|\varphi_n(x_n)-\varphi_n(y_n)\|e_n\|\\
&\leq \|\sum_{n\in\N}L\|x_n-y_n\|e_n\|+\|\sum_{n\in\N}\frac{L}{2^n}e_n\|\leq L\|x-y\|+L.
\end{align*}\hfill

\noindent So $\Phi$ is coarse.

Assume $\varphi$ is uniformly continuous, let us show that $\Phi$ is also uniformly continuous. Fix $\eps>0$. Pick $N\in\N$ such that $\sum_{n>N}2^{-n}<\eps /3L$. Choose $\delta>0$ such that $\delta< \eps /3L$, and $\|\varphi_n(x_n)-\varphi_n(y_n)\|<\eps /3N$, for all $n\leq N$, and all $x_n,y_n\in X$ such that $\|x_n-y_n\|<\delta$.  Then, if $\|x-y\|<\delta$, we have

\begin{align*}\|\Phi(x)-&\Phi(y)\|\\
 &\leq\|\sum_{n\leq N}\|\varphi_n(x_n)-\varphi_n(y_n)\|e_n\|+ \|\sum_{n>N}L\|x_n-y_n\|e_n\|+\|\sum_{n>N}\frac{L}{2^n}e_n\|\\ &\leq \eps/3+\eps/3+\eps/3=\eps
\end{align*}\hfill

\noindent This shows that $\Phi$ is uniformly continuous.

Say  $\varphi$ is a uniform homeomorphism. Notice that $\varphi_n^{-1}(\cdot)=2^{-n}\varphi^{-1}(2^n\cdot)$, therefore, $\Phi^{-1}(\cdot)=(\varphi^{-1}_n(\cdot))_n$, and, by the same arguments as above, $\Phi^{-1}$ is uniformly continuous. Hence,  $(\oplus X)_\mathcal{E}$ and $(\oplus Y)_\mathcal{E}$ are uniformly homeomorphic.

If $\varphi$ is a coarse equivalence, let $\psi:Y\to X$ be a coarse inverse for $\varphi$ (see Subsection \ref{coarsedef}, Remark \ref{remarkcoarseinv}).  Let $\psi_n(\cdot)=2^{-n}\psi(2^n\cdot)$, and $\Psi=(\psi_n)_n$. Then, by the same arguments above, $\Psi$ is coarse. One can easily see that $\Phi$ and $\Psi$ are coarse inverses of each other, so we are done.

The case of simultaneous homeomorphic and coarse equivalences follows analogously.
\qed\\

%\begin{align*}
%\|x-\Psi(\Phi(x))\|&\leq \sum_{n\in\N}\|x_n-2^{-n}\psi(\varphi(2^nx_n)\|\\
%&\leq\sum_{n\in\N} 2^{-n}\|2^nx_n- \psi(\varphi(2^nx_n)\|\leq \eps.
%\end{align*}\hfill

The proof above actually gives us the following slightly stronger result.

\begin{thm}\label{geral}
Let $(X_n)_n$ and $(Y_n)_n$ be sequences of Banach spaces, and let $\varphi_n:X_n\to Y_n$ be a coarse equivalence (resp. uniform homeomorphism, or simultaneously homeomorphic and coarse equivalence), for each $n\in\N$. Let $\mathcal{E}$ be a normalized $1$-unconditional  basic sequence. Assume that  

$$\sup_n\lim_{t\to\infty}\frac{\omega_{\varphi_n}(t)}{t}<\infty\ \  \text{ and }\ \ \inf_n\lim_{t\to\infty}\frac{\rho_{\varphi_n}(t)}{t}>0.$$\hfill

\noindent Then $(\oplus_n X_n)_\mathcal{E}$ and $(\oplus_n Y_n)_\mathcal{E}$ are  coarsely equivalent (resp. uniformly homeomorphic, or simultaneously homeomorphically and coarsely equivalent).
\end{thm}

\begin{proof}
Let us work with the uniform homeomorphism case. Without loss of generality, we assume that $\varphi_n(0)=0$, for all $n\in\N$. Let $L>0$ be large enough so that $\lim_{t\to\infty}\omega_{\varphi_n}(t)/t<L$, and $\lim_{t\to\infty}\rho_{\varphi_n}(t)/t>1/L$, for all $n\in\N$. For each $n\in\N$, pick $t_n>0$ such that  $\omega_{\varphi_n}(t)<Lt$, and $\rho_{\varphi_n}(t)>t/L$, for all $n\in\N$, and all $t\geq t_n$. Then, 

$$\omega_{\varphi_n}(t)<Lt+Lt_n\ \ \text{ and }\ \ \rho_{\varphi_n}(t)>\frac{1}{L}t-\frac{1}{L}t_n,$$\hfill

\noindent for all $n\in\N$, and all $t>0$. Hence, it is easy to check that $\omega_{\varphi^{-1}_n}(t)<Lt+t_n$, for all $n$, and all $t>0$. Setting 

$$\tilde{\varphi}_n(x)=\frac{1}{2^nt_n}\varphi_n(2^nt_nx),$$\hfill

\noindent we have that each $\tilde{\varphi}_n$ is a uniform homeomorphism between $X_n$ and $Y_n$, and that  

$$\omega_{\tilde{\varphi}_n}(t)<Lt+\frac{L}{2^n}\ \ \text{ and } \ \ \ \omega_{\tilde{\varphi}^{-1}_n}(t)<Lt+\frac{1}{2^n},$$\hfill

\noindent for all $n\in\N$, and all $t>0$. The proof now follows analogously the proof of Theorem  \ref{trikal}. For the coarse equivalence case we only need to work with the coarse inverses of $\varphi_n$'s instead of its inverses, and proceed similarly.
 \end{proof}

\begin{cor}
Let $X$ and $Y$ be coarsely equivalent (resp. uniformly homeomorphic, or simultaneously homeomorphically and coarsely equivalent) Banach spaces, then $\ell_p(X)$ and $\ell_p(Y)$ are coarsely equivalent (resp. uniformly homeomorphic, or simultaneously homeomorphically and coarsely equivalent).
\end{cor}

\begin{remark}
The conditions on $\omega_{\varphi_n}$ and $\rho_{\rho_n}$ in Theorem \ref{geral} cannot be omitted.  Indeed, let  $q_X=\inf\{q\in [2,\infty)\mid X\text{ has cotype }q\}$, for any Banach space $X$. Then, by Theorem 1.11 (resp. Theorem 1.9) of \cite{MN},  if a Banach space $X$ coarsely (resp. uniformly) embeds into a Banach space $Y$ with nontrivial type, then $q_X\leq q_Y$. Therefore, $(\oplus_n \ell^n_{\infty})_2$ does not coarsely (resp. uniformly) embed into $(\oplus_n \ell^n_{2})_2\cong \ell_2$, as $q_{(\oplus_n \ell^n_{\infty})_2}=\infty$ and $q_{\ell_2}=2$. 
\end{remark}

Clearly, the method above gives us that, if $X$ coarse Lipschitz embeds into $Y$, then $(\oplus X)_\mathcal{E}$ coarse Lipschitz embeds into $(\oplus Y)_\mathcal{E}$. But we do not know the answer for this question if we do not have a linear lower bound for the compression modulus of the embedding. We end this section with a natural question.

\begin{problem}
Let $X$ and $Y$ be Banach spaces, and let $\mathcal{E}$ be a $1$-unconditional basic sequence. Assume that  $X$ coarsely (resp. uniformly) embeds into $Y$. Does it follow that $(\oplus X)_\mathcal{E}$ coarsely (resp. uniformly) embeds into $(\oplus Y)_\mathcal{E}$?
\end{problem}

Notice that the answer to the problem above is trivially yes if $\mathcal{E}$ is the standard basis of $c_0$. Indeed, if $\varphi:X\to Y$ is a uniform embedding and $\varphi(0)=0$, then  $\Phi=(\varphi)_n$ is a uniform embedding of $(\oplus X)_{c_0}$ into $(\oplus Y)_{c_0}$. If $\varphi$ is a coarse embedding, then $\Phi=(\varphi)_n$ does not need to be well defined, so the same argument does not work. However, without loss of generality, we can assume that $\varphi(x)=0$, for  all $x\in B_X$. Then, the map $\Phi=(\varphi)_n$ is well defined, and it is a coarse embedding.

\begin{remark} We should notice that, in \cite{Ka2}, Kalton only deals with what he calls ``coarse homeomorphisms$"$, i.e., a coarse equivalence which is also a bijection. However it is easy to show that $X$ and $Y$ are coarsely homeomorphic  if and only if  $X$ and $Y$ are coarsely equivalent, for all Banach spaces $X$ and $Y$.  This follows from the easy fact that if $X$ and $Y$ are coarsely equivalent, then $X$ and $Y$ have the same density character, which equals the cardinality of any net in $X$ and $Y$ (for separable Banach spaces this follows from Proposition \ref{measurable}).
\end{remark}

\noindent \textbf{Acknowledgments:} The author would like to thank his adviser C. Rosendal for all the help and attention he gave to this paper. The author would also like to thank the  anonymous referee for all their comments and corrections.

\end{document}